\documentclass[11pt]{article}
\usepackage{latexsym}
\usepackage{amsmath, amsthm, amsfonts}
\usepackage{graphics}
\usepackage{graphicx}
\usepackage{color}
\usepackage[table,xcdraw]{xcolor}
\usepackage{float}
\usepackage{array}
\usepackage{hyperref}

\newcommand{\abs}[1]{\left\lvert#1\right\rvert}
\newtheorem{theorem}{Theorem}[section]

\newtheorem{conjecture}{Conjecture}[section]

\def\smallskip{\addvspace{\smallskipamount}}
\def\medskip{\addvspace{\medskipamount}}
\def\bigskip{\addvspace{\bigskipamount}}
\def\makefootline{\baselineskip=24pt \line{\the\footline}}
\def\pagecontents{\ifvoid\topins\else\unvbox\topins\fi
   \dimen@=\dp255 \unvbox255
   \ifvoid\footins\else
      \vskip\skip\footins \footnoterule \unvbox\footins\fi
     \ifr@ggedbottom \kern-\dimen@ \vfil \fi}
\def\footnoterule{\kern-3pt\hrule width 2truein \kern 2.6pt}

\begin{document}

\title{{Dynamics of $\displaystyle{z_{n+1}=\frac{\alpha + \alpha z_{n}+\beta z_{n-1}}{1+z_{n}}}$} in Complex Plane}

\author{Sk. Sarif Hassan\\
  \small {\emph{Department of Mathematics}}\\
  \small {\emph{University of Petroleum and Energy Studies}}\\
  \small {\emph{Bidholi, Dehradun, India}}\\
  \small Email: {\texttt{\textcolor[rgb]{0.00,0.00,1.00}{s.hassan@ddn.upes.ac.in}}}\\
}

\maketitle
\begin{abstract}
\noindent The dynamics of the second order rational difference equation $\displaystyle{z_{n+1}=\frac{\alpha + \alpha z_{n}+\beta z_{n-1}}{1+z_{n}}}$ with complex parameters $\alpha$, $\beta$ and arbitrary complex initial conditions is investigated. The same difference equation is well studied with positive real parameters and initial values and one of the main results on trichotomy of parameters is revisited in the complex set-up and the similar result is found to be true. In addition, the chaotic solutions of the equation is fetched out in the complex set-up which was absent in the real scenario.  \\
\end{abstract}

\begin{flushleft}\footnotesize
{Keywords: Rational difference equation, Local asymptotic stability, Chaotic trajectory and Periodicity. \\

}
\end{flushleft}

\section{Introduction and Background}

Consider the equation

\begin{equation}
\displaystyle{z_{n+1}=\frac{\alpha + \alpha z_{n}+\beta z_{n-1}}{1+z_{n}}},  n=0,1,2,\ldots
\label{equation:total-equationA}
\end{equation}%
where the parameters $\alpha$ and $\beta$ and the initial conditions $z_{-1}$  and $z_{0}$ are arbitrary complex numbers.

\addvspace{\bigskipamount}
\noindent
This second order rational difference equation Eq.(\ref{equation:total-equationA}) is studied when the parameters $\alpha$ and $\beta$ a real number and the initial conditions are non-negative real numbers in \cite{S-H}, \cite{K-L} \& \cite{C-L}. In this present article it is an attempt to understand the dynamics by means of local asymptotic stability of equilibriums, periodicity, chaoticity of trajectories in the complex plane.\\

\noindent
In the real parameters and initial values, the main result on trichotomy character on the parameters is given below. It is noted that $\alpha$ and $\beta$ and the initial values are positive real numbers \cite{Gibb}.
\begin{enumerate}
  \item Every solution of Eq.(\ref{equation:total-equationA}) has a finite limit if and only if $\beta < \alpha+1$.
  \item Every solution of Eq.(\ref{equation:total-equationA}) converges to period two solution if and only if $\beta = \alpha+1$.
  \item Every solution of Eq.(\ref{equation:total-equationA}) is unbounded if and only if $\beta > \alpha+1$.\\
\end{enumerate}

\noindent
Here our main aim is to study the dynamics of the (\ref{equation:total-equationA}) under the assumption that the parameters and the initial conditions are arbitrary complex numbers.

\section{Local Asymptotic Stability of the Equilibriums}
In this section, the local stability character of the equilibria of Eq.(\ref{equation:total-equationA}) is established \cite{S-E}.

\subsection{Local Asymptotic Stability of $z_{n+1}=\frac{\alpha + \alpha z_{n}+\beta z_{n-1}}{1+z_{n}}$}
The equilibrium points of Eq.(\ref{equation:total-equationA}) are the solutions of the quadratic equation
\[
\bar{z}=\frac{\alpha+\alpha \bar{z}+\beta \bar{z}}{1+\bar{z}}
\]

\noindent
Note that if $\alpha=0$, $\bar{z}_{\pm}=0$ and $\alpha+\beta-1$ are two equilibriums of the equation Eq.(\ref{equation:total-equationA}).\\

\noindent
The Eq.(1) has the two equilibria points
$ \bar{z}_{1,2} =\frac{1}{2} \left(-1+\alpha +\beta -\sqrt{1+2 \alpha +\alpha ^2-2 \beta +2 \alpha  \beta +\beta ^2}\right)$ and $\frac{1}{2} \left(-1+\alpha +\beta +\sqrt{1+2 \alpha +\alpha ^2-2 \beta +2 \alpha  \beta +\beta ^2}\right)$ respectively when $\alpha \neq 0$.
\noindent
The linearized equation of the rational difference equation Eq.(\ref{equation:total-equationA}) with respect to the equilibrium point $ \bar{z}_{1} = \frac{1}{2} \left(-1+\alpha +\beta -\sqrt{1+2 \alpha +\alpha ^2-2 \beta +2 \alpha  \beta +\beta ^2}\right)$ is

\tiny

\begin{equation}
\label{equation:linearized-equation}
\displaystyle{
z_{n+1} + \frac{2 \beta  \left(-1+\alpha +\beta -\sqrt{(1+\alpha )^2+2 (-1+\alpha ) \beta +\beta ^2}\right)}{\left(1+\alpha +\beta -\sqrt{(1+\alpha )^2+2 (-1+\alpha ) \beta +\beta ^2}\right)^2} z_{n} -\frac{1}{2} \left(1+\alpha +\beta +\sqrt{(1+\alpha )^2+2 (-1+\alpha ) \beta +\beta ^2}\right) z_{n-1}=0,  n=0,1,\ldots
}
\end{equation}

\small

\noindent
with associated characteristic equation
\tiny
\begin{equation}
\lambda^{2} + \frac{2 \beta  \left(-1+\alpha +\beta -\sqrt{(1+\alpha )^2+2 (-1+\alpha ) \beta +\beta ^2}\right)}{\left(1+\alpha +\beta -\sqrt{(1+\alpha )^2+2 (-1+\alpha ) \beta +\beta ^2}\right)^2} \lambda -\frac{1}{2} \left(1+\alpha +\beta +\sqrt{(1+\alpha )^2+2 (-1+\alpha ) \beta +\beta ^2}\right) = 0.
\end{equation}

\addvspace{\bigskipamount}
\small

\noindent
By the \emph{Clark's theorem}, the equilibrium $\bar{z}_{1}=\frac{1}{2} \left(-1+\alpha +\beta -\sqrt{1+2 \alpha +\alpha ^2-2 \beta +2 \alpha  \beta +\beta ^2}\right)$  is \emph{locally asymptotically stable} if the sum of the modulus of two coefficients of the characteristic equation Eq. (3) is less than $1$. Therefore the condition for local asymptotic stability of the equilibrium would be $$\abs{\frac{2 \beta  \left(-1+\alpha +\beta -\sqrt{(1+\alpha )^2+2 (-1+\alpha ) \beta +\beta ^2}\right)}{\left(1+\alpha +\beta -\sqrt{(1+\alpha )^2+2 (-1+\alpha ) \beta +\beta ^2}\right)^2}} + \abs{\frac{1}{2} \left(1+\alpha +\beta +\sqrt{(1+\alpha )^2+2 (-1+\alpha ) \beta +\beta ^2}\right)}<1$$

\noindent
It is noted that that the conditional inequality for local asymptotic stability of the equilibrium attains maximum $19.7392$ at $\alpha=0.93601--0.293238i$ and $\beta=-0.00195532-0.211385i$. Also the minimum value attain by the inequality is $1.28533$ at $\alpha=0.86278+0.446302i$ and $\beta=-0.0309069+0.749819i$.\\ \\

\noindent
This numerical result suggests that there does not exist any $\alpha$ and $\beta$ so that the inequality holds. Therefore the equilibrium $\bar{z}_{1}=\frac{1}{2} \left(-1+\alpha +\beta -\sqrt{1+2 \alpha +\alpha ^2-2 \beta +2 \alpha  \beta +\beta ^2}\right)$  is \emph{unstable} indeed. \\ \\

\noindent
The linearized equation of the rational difference equation Eq.(\ref{equation:total-equationA}) with respect to the equilibrium point $ \bar{z}_{2} =\frac{1}{2} \left(-1+\alpha +\beta +\sqrt{1+2 \alpha +\alpha ^2-2 \beta +2 \alpha  \beta +\beta ^2}\right)$ is

\tiny
\begin{equation}
\label{equation:linearized-equation}
\displaystyle{
z_{n+1} + \frac{2 \beta  \left(-1+\alpha +\beta +\sqrt{(1+\alpha )^2+2 (-1+\alpha ) \beta +\beta ^2}\right)}{\left(1+\alpha +\beta +\sqrt{(1+\alpha )^2+2 (-1+\alpha ) \beta +\beta ^2}\right)^2} z_{n} - \frac{2 \beta }{1+\alpha +\beta +\sqrt{(1+\alpha )^2+2 (-1+\alpha ) \beta +\beta ^2}} z_{n-1}=0,  n=0,1,\ldots
}
\end{equation}
\small
\noindent
with associated characteristic equation
\tiny
\begin{equation}
\lambda^{2} + \frac{2 \beta  \left(-1+\alpha +\beta +\sqrt{(1+\alpha )^2+2 (-1+\alpha ) \beta +\beta ^2}\right)}{\left(1+\alpha +\beta +\sqrt{(1+\alpha )^2+2 (-1+\alpha ) \beta +\beta ^2}\right)^2} \lambda -\frac{2 \beta }{1+\alpha +\beta +\sqrt{(1+\alpha )^2+2 (-1+\alpha ) \beta +\beta ^2}} = 0.
\end{equation} \\

\addvspace{\bigskipamount}

\small
\begin{theorem}
The equilibriums $\bar{z}_{2}=\frac{1}{2} \left(-1+\alpha +\beta +\sqrt{1+2 \alpha +\alpha ^2-2 \beta +2 \alpha  \beta +\beta ^2}\right)$ of Eq.(\ref{equation:total-equationA}) is locally asymptotically stable if $$\abs{\frac{2 \beta  \left(-1+\alpha +\beta +\sqrt{(1+\alpha )^2+2 (-1+\alpha ) \beta +\beta ^2}\right)}{\left(1+\alpha +\beta +\sqrt{(1+\alpha )^2+2 (-1+\alpha ) \beta +\beta ^2}\right)^2}}+\abs{\frac{2 \beta }{1+\alpha +\beta +\sqrt{(1+\alpha )^2+2 (-1+\alpha ) \beta +\beta ^2}}}<1$$
\end{theorem}

\begin{proof}
Proof the theorem follows from the \emph{Clark's theorem} of local asymptotic stability of the equilibriums. Hence by \emph{Clark's theorem}, the condition for the local asymptotic stability would be $$\abs{\frac{2 \beta  \left(-1+\alpha +\beta +\sqrt{(1+\alpha )^2+2 (-1+\alpha ) \beta +\beta ^2}\right)}{\left(1+\alpha +\beta +\sqrt{(1+\alpha )^2+2 (-1+\alpha ) \beta +\beta ^2}\right)^2}}+\abs{\frac{2 \beta }{1+\alpha +\beta +\sqrt{(1+\alpha )^2+2 (-1+\alpha ) \beta +\beta ^2}}}<1$$

\end{proof}

\noindent
It is noted that that the conditional inequality for local asymptotic stability of the equilibrium attains maximum $2.69519$ at $\alpha=-0.86278+0.446302i$ and $\beta=-0.0309069+0.749819i$. Also the minimum value attain by the inequality is $(3.211 \times 10^{-6})$ which is approximately $0$ at $\alpha=1-(6.4974 \times 10^{-8})$ and $\beta=(-3.02771 \times 10^{-6}) -(3.02772 \times 10^{-6})$.\\

\noindent
This numerical observation assures that there exist $\alpha$ and $\beta$ such that the inequality given in \emph{Theorem 2.1} condition holds good.

\subsubsection{Local Asymptotic Stability in the case $\alpha=\beta$}

When two parameters $\alpha$ and $\beta$ are equal then the difference equation Eq.(\ref{equation:total-equationA}) would reduce to

\begin{equation}
\label{equation:linearized-equation11}
\displaystyle{
z_{n+1} =\frac{\alpha (1+z_n+z_{n-1})}{1+z_{n}},  n=0,1,\ldots
}
\end{equation}

\noindent
The equilibriums $z_{\pm}$ of the equation Eq.(\ref{equation:linearized-equation11}) are $\frac{1}{2} \left(-1+2 \alpha -\sqrt{1+4 \alpha ^2}\right)$ and $\frac{1}{2} \left(-1+2 \alpha +\sqrt{1+4 \alpha ^2}\right)$ respectively.

\noindent
In the similar fashion as we did in the above case in section 2.1, the equilibrium $z_{+}= \frac{1}{2} \left(-1+2 \alpha -\sqrt{1+4 \alpha ^2}\right)$ of the difference equation Eq.(\ref{equation:linearized-equation11}) is locally asymptotically stable if $\abs{\frac{1+\sqrt{1+4 \alpha ^2}+\alpha  \left(1+2 \alpha +\sqrt{1+4 \alpha ^2}\right)}{2 \alpha }}+\abs{\frac{1}{2} \left(1+2 \alpha +\sqrt{1+4 \alpha ^2}\right)}<1$.\\

\noindent
But it is seen as before numerically that the minimum value attain by the inequality is greater than $1$. Therefore there does not exist an $\alpha$ so that the inequality does hold. So the equilibrium $z_{+}= \frac{1}{2} \left(-1+2 \alpha -\sqrt{1+4 \alpha ^2}\right)$ of the difference equation Eq.(\ref{equation:linearized-equation11}) is \emph{unstable}. \\

\noindent
Similarly, the equilibrium $z_{-} =\frac{1}{2} \left(-1+2 \alpha +\sqrt{1+4 \alpha ^2}\right)$ is locally asymptotically stable if $\abs{\frac{2 \alpha  \left(-1+2 \alpha +\sqrt{1+4 \alpha ^2}\right)}{\left(1+2 \alpha +\sqrt{1+4 \alpha ^2}\right)^2}}+\abs{\frac{1}{2}+\alpha -\frac{1}{2} \sqrt{1+4 \alpha ^2}}<1$.\\ \\

\noindent
It is observed numerically that the maximum value of the inequality is 1.15792 for $\alpha=-0.535769+-0.13703i$ and the minimum value is $4 \times 10^{-6}$ (close to zero) for $\alpha$ approximately equal to zero. This ensures that there exists $\alpha$ such that the inequality does hold good. So the local asymptotic stability at the $z_{-} =\frac{1}{2} \left(-1+2 \alpha +\sqrt{1+4 \alpha ^2}\right)$ is assured.\\

\noindent
Here are few example cases for the local asymptotic stability of the equilibriums of the Eq.(\ref{equation:linearized-equation11}).\\

\noindent
In the equation Eq.(\ref{equation:linearized-equation11}), for $\alpha=\beta=1+i$, the equilibriums are $\frac{1}{2} \left((1+2 i)-\sqrt{1+8 i}\right)$ and $\frac{1}{2} \left((1+2 i)+\sqrt{1+8 i}\right)$. For the equilibrium $\frac{1}{2} \left((1+2 i)-\sqrt{1+8 i}\right)$, the coefficients of the characteristic equation Eq. (3) are $\left(\frac{7}{4}+\frac{3 i}{4}\right)+\frac{1}{2} \sqrt{14+\frac{29 i}{2}}$ and $\frac{1}{2} \left((3+2 i)+\sqrt{1+8 i}\right)$ with modulus $4.14869$ and $3.21522$. Therefore the condition as stated in the \emph{Theorem 2.1} does not hold. Therefore the  equilibrium $\frac{1}{2} \left((1+2 i)-\sqrt{1+8 i}\right)$ is \emph{unstable}.\\

\noindent
In Eq.(\ref{equation:linearized-equation11}), when the parameters $\alpha$ and $\beta$ both equal and equal to $1+i$, then for the equilibrium $\frac{1}{2} \left((1+2 i)+\sqrt{1+8 i}\right)$, the coefficients of the characteristic equation Eq. (3) are $\left(\frac{7}{4}+\frac{3 i}{4}\right)-\frac{1}{2} \sqrt{14+\frac{29 i}{2}}$ and $\frac{1}{2} \left((3+2 i)-\sqrt{1+8 i}\right)$ with modulus $0.340882$ and $0.439849$. Therefore the condition as stated in the \emph{Theorem 2.1} holds good. Therefore the  equilibrium $\frac{1}{2} \left((1+2 i)+\sqrt{1+8 i}\right)$ is \emph{locally asymptotically stable}.\\

\section{Boundedness and Unbounded Solutions}
Here problem is to determine an open ball $B(0, \epsilon) \in \mathbb{C}$ such that if $z_n \in B(0, \epsilon)$ and $z_{n-1} \in B(0, \epsilon)$ then $z_{n+1} \in B(0, \epsilon)$ for all $n \geq 0$.

\begin{theorem}
 For the difference equation Eq.(\ref{equation:total-equationA}), for every $\epsilon >0$, if $z_n$ and $z_{n-1}$ $\in B(0, \epsilon)$ then $z_{n+1} \in B(0, \epsilon)$ provided $$\abs{\alpha}+\abs{\beta} \leq 1- \epsilon -\frac{\abs{\alpha}}{\epsilon}$$
\end{theorem}

\noindent

\begin{proof}
Let $\{z_{n}\}$ be a solution of the equation Eq.(\ref{equation:total-equationA}). Let $\epsilon>0$ be any arbitrary real number. Consider $z_n, z_{n-1} \in B(0, \epsilon)$. We need to find out an $\epsilon$ such that $z_{n+1}\in B(0, \epsilon)$ for all $n$. It is follows from the Eq.(\ref{equation:total-equationA}) that for any $\epsilon >0$, using Triangle inequality for $$\abs{z_{n+1}}=\abs{\frac{\alpha+\alpha z_n+\beta z_{n-1}}{1+z_{n}}} \leq \frac{\abs{\alpha}+\abs{\alpha}\epsilon+\abs{\beta}\epsilon}{1-\epsilon}$$ In order to ensure that $\abs{z_{n+1}}<\epsilon$, it is required to be $$\frac{\frac{\abs{\alpha}}{\epsilon}+\abs{\alpha}+\abs{\beta}}{1-\epsilon}<1$$ That is $\abs{\alpha}+\abs{\beta} \leq 1- \epsilon -\frac{\abs{\alpha}}{\epsilon}$. Hence it is proved.

\end{proof}

\noindent
In contrast, we found many parameters $\alpha$ and $\beta$ where the solutions are unbounded for any initial values. A brief list is given in the following Table 1.

\begin{table}[H]

\begin{tabular}{| m{3cm} || m{5cm} || m{5cm} |}
\hline
\centering  \textbf{Serial Number} &
\begin{center}
\textbf{Parameters} $\alpha$, $\beta$
\end{center}
 & \textbf{Condition:} $\abs{\beta} > \abs{\alpha+1}$ \\
\hline
\centering  1 &
\begin{center}
$\alpha=0.09594 + 0.06016i$, $\beta=0.81950 + 0.77147i$
\end{center} & $\abs{\alpha+1}=1.0976$ \& $\abs{\beta}=1.1255$\\
\hline
\centering  2 &
\begin{center}
$\alpha=0.02817 + 0.70953i$, $\beta=0.90515 + 0.86582i$
\end{center} & $\abs{\alpha+1}=1.2492$ \& $\abs{\beta}=1.2526$\\
\hline
\centering  3 &
\begin{center}
$\alpha=0.10288 + 0.23615i$, $\beta=0.97178 + 0.69889i$
\end{center} & $\abs{\alpha+1}=1.1279$ \& $\abs{\beta}=1.1970$ \\
\hline
\centering  4 &
\begin{center}
$\alpha=0.04069 + 0.32936i$, $\beta=0.70934 + 0.8603i$
\end{center} & $\abs{\alpha+1}=1.0916$ \& $\abs{\beta}=1.1150$\\
\hline
\centering  5 &
\begin{center}
$\alpha=0.09663 + 0.26539i$, $\beta=0.791945 + 0.93686i$
\end{center} & $\abs{\alpha+1}=1.1283$ \& $\abs{\beta}= 1.2267$\\
\hline
\centering  6 &
\begin{center}
$\alpha=0.17020 + 0.30234i$, $\beta=0.96143 + 0.7836i$
\end{center} & $\abs{\alpha+1}=1.2086$ \& $\abs{\beta}= 1.2403$\\
\hline
\centering  7 &
\begin{center}
$\alpha=0.24314 + 0.15415i $, $\beta=0.95641 + 0.9356i$
\end{center} & $\abs{\alpha+1}=1.2527$ \& $\abs{\beta}=1.3379$\\
\hline
\centering  8 &
\begin{center}
$\alpha=0.78962 + 0.79918i$, $\beta=2\alpha$
\end{center} &  $\abs{\alpha+1}=1.9600$ \& $\abs{\beta}= 2.2469$\\
\hline
\centering  9 &
\begin{center}
$\alpha=40+33i$, $\beta=27+77i$
\end{center} & $\abs{\alpha+1}=52.6308$ \& $\abs{\beta}=81.5966$\\
\hline
\centering  10 &
\begin{center}
$\alpha=60+4i$, $\beta=89+86i$
\end{center} & $\abs{\alpha+1}=61.1310$ \& $\abs{\beta}= 123.7619$\\
\hline

\end{tabular}
\caption{Parameters $\alpha$ and $\beta$ for which solutions possess unbounded solutions of Eq.(\ref{equation:total-equationA}) for arbitrary initial values.}
\label{Table:}
\end{table}
\noindent
Here in the Table 1, the condition $\abs{\beta} > \abs{\alpha+1}$ is satisfied for all the ten cases. Therefore it is verified numerically that the similar result to the real set up in the case of unboundedness of solutions of the equation Eq.(\ref{equation:total-equationA}) in the complex set-up is followed also. The following figure includes the orbit plots of the four unbounded solutions of the Eq.(\ref{equation:total-equationA}) for which the parameters are given in the Table 1 in four rows from the serial number 1 to 4. \\

\begin{figure}[H]
      \centering

      \resizebox{12cm}{!}
      {
      \begin{tabular}{c c}
      \includegraphics [scale=5]{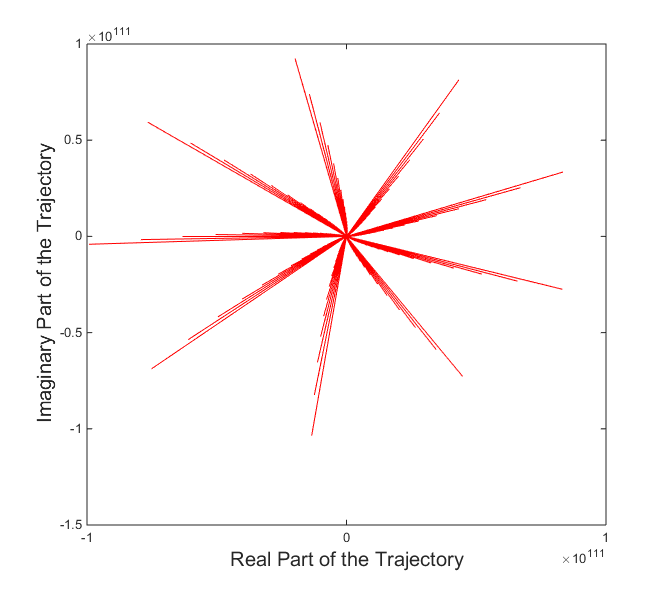}
      \includegraphics [scale=5]{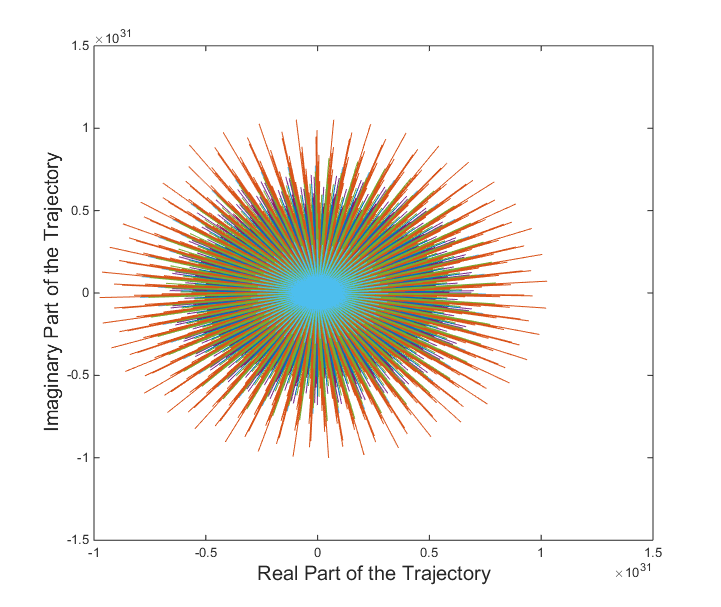}\\
      \includegraphics [scale=5]{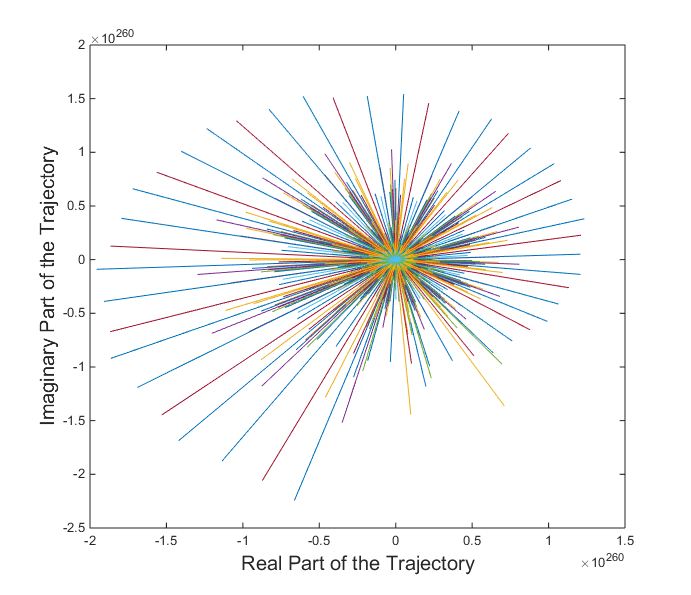}
      \includegraphics [scale=5]{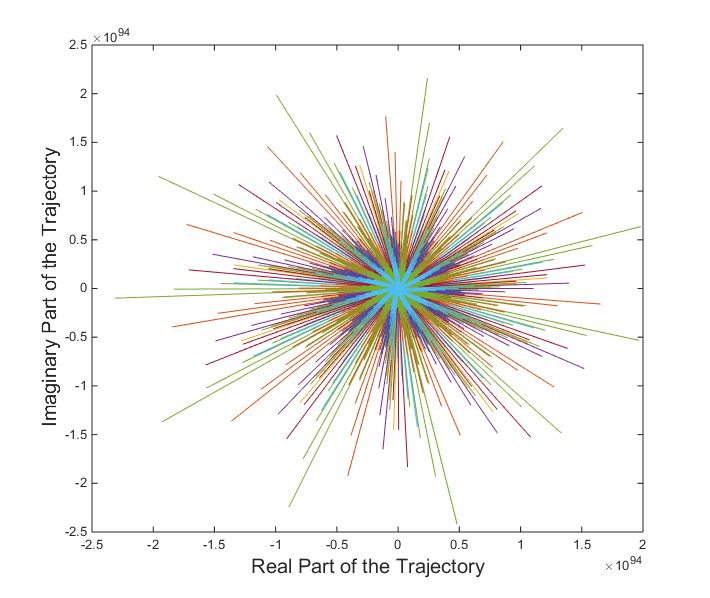}\\
\end{tabular}
      }
\caption{Unbounded Orbit for the Parameters (Serial 1-4) as given in Table.1, Serial:1-Left upper most, Serial2:Right upper most, Serial:3-Left lower most and Serial:4- Right lower most.}
      \begin{center}

      \end{center}
      \end{figure}
\noindent
In the Fig. 1, the unbounded solution for the specified $\alpha$, $\beta$ with any 100 initial values (one initial value for Serial-1) are plotted. It is seen that all the solutions in these four cases are unbounded. \\
\section{Periodic of Solutions}

\label{section:periodicity}

\noindent
We shall first look for the prime period two solutions of the three difference equation Eq.(\ref{equation:total-equationA}) and their corresponding local stability analysis.\\

\noindent
Let $\ldots, \phi ,~\psi , ~\phi , ~\psi ,\ldots$, $\phi  \neq \psi $ be a prime period two solution of the difference equation Eq.(\ref{equation:total-equationA}). Then $\phi=\frac{\alpha+\alpha\phi+\beta\psi}{1+\psi}$ and $\psi=\frac{\alpha+\alpha\psi+\beta\phi}{1+\phi}$. By solving these two equations we get two equilibriums only. That is, there is no prime period solution of the difference equation Eq.(\ref{equation:total-equationA}) other than the equilibriums which we found earlier.\\
But in case $\beta-\alpha=1$, there are prime period two solutions exist for the equation Eq.(\ref{equation:total-equationA}). A list is given for the computational evidence.

\begin{table}[H]

\begin{tabular}{| m{4cm} || m{5cm} || m{5cm}|}
\hline
\centering   \textbf{Parameters}: $\alpha$, $\beta$ & \textbf{Initial Values} &
\begin{center}
\textbf{Periodic Solutions}
\end{center}\\
\hline
\centering $\alpha=0.8407+ 0.2542i$, $\beta=\alpha+1$ &
\begin{center}
$z_0=55 + 14i$ $z_1=15+26i$
\end{center} & 0.8648 + 0.1258i, -2.191 + 12.3691i\\
\hline
\centering $\alpha=0.1966 + 0.2511i$, $\beta=\alpha+1$ &
\begin{center}
$z_0=82 + 24i$ $z_1=93+25i$
\end{center} & 0.2021 + 0.2559i, 48.7958 + 20.6764i \\
\hline
\centering $\alpha=62+47i$, $\beta=\alpha+1$ &
\begin{center}
$z_0=62 + 47i$ $z_1=35+83i$
\end{center} & 0.6105 + 0.5405i, 7.2939 + 50.1430i  \\
\hline
\centering $\alpha=0.3804 + 0.5678i$, $\beta=\alpha+1$ &
\begin{center}
$z_0=92 + 28i$ $z_1=76+76i$
\end{center} & 0.3970 + 0.5736i, 30.1391 + 50.3749i\\
\hline

\end{tabular}
\caption{For $\beta=\alpha+1$, Prime Period 2 solutions of the equation Eq.(\ref{equation:total-equationA}) for different choice of parameters and initial values.}
\label{Table:}
\end{table}
\noindent
Therefore the solution of the difference equation Eq.(\ref{equation:total-equationA}) converges to the periodic solution of period 2 in the case of $\beta=\alpha+1$ that is $\abs{\beta}=\abs{\alpha+1}$ is holding well which was exactly same condition in the case of real parameters.\\

\noindent
Let us now establish the fact by proving the following theorem.\\

\begin{theorem}
For $\beta=\alpha+1$, the solution of the difference equation Eq.(\ref{equation:total-equationA}) converges to the periodic solution of period 2.
\end{theorem}

\begin{proof}
When $\beta=\alpha+1$, then all prime period 2 solutions of Eq.(\ref{equation:total-equationA}) $\dots, \phi, \psi, \phi, \psi, \dots$ are the solution of the equation

\begin{equation}\label{sa}
  \alpha+\alpha (\phi+\psi)-\phi\psi=0
\end{equation}
\noindent
Keeping in mind the equation Eq. (\ref{sa}) we set $$\mathfrak{J}_n=\alpha+\alpha(\phi+\psi)-\phi\psi$$ for $n>0$.\\

\noindent
The following identities are holding well for $n >0$

\begin{equation}\label{sa1}
\mathfrak{J}_{n+1}=\frac{\alpha+1}{1+z_n}\mathfrak{J}_n
\end{equation}

\begin{equation}\label{sa2}
z_{n+1}-z_{n-1}=\frac{\mathfrak{J}_n}{1+z_n}
\end{equation}

\noindent
and for $n \geq 1$
\begin{equation}\label{sa3}
z_{n+1}-z_{n-1}=\frac{\alpha+1}{1+z_n} (z_{n}-z_{n-2})
\end{equation}

\noindent
and therefore, for all $n \geq 1$
\begin{equation}\label{sa4}
z_{n+1}-z_{n-1}=(z_1-z_{-1})\prod_{k=1}^{n} \frac{\alpha+1}{1+z_k}
\end{equation}

\noindent
From \ref{sa1}, it is seen that $\mathfrak{J}_n$ remains constant along each solution and so it is now following from \ref{sa2} that for each solution $\{z_n\}_n$ of the equation  Eq.(\ref{equation:total-equationA}), exactly one of the following two conditions would be true. \\
\noindent
Either $$\abs{z_{n+1}-z_{n-1}}=0$$ or $$\abs{z_{n+1}-z_{n-1}} \neq 0$$

\noindent
If the the either condition holds then the proof is immediate. If the other condition holds, then solution would converges to the period $2$ solutions provided the solution is bounded.\\
\noindent
Hence the theorem is proved.\\
\end{proof}

\noindent
For $100$ different initial values and for different parameters, solutions are generated and corresponding periodic plot of $50$ iterations are plotted in the Fig.2. Different colors exhibits different periodic trajectories. From the figure it is evident that all $100$ solutions are periodic and of period 2. This observation led to make the following conjecture which is a result in the case of real parameters and initial values.

\begin{figure}[H]
      \centering

      \resizebox{10cm}{!}
      {
      \begin{tabular}{c}
      \includegraphics [scale=15]{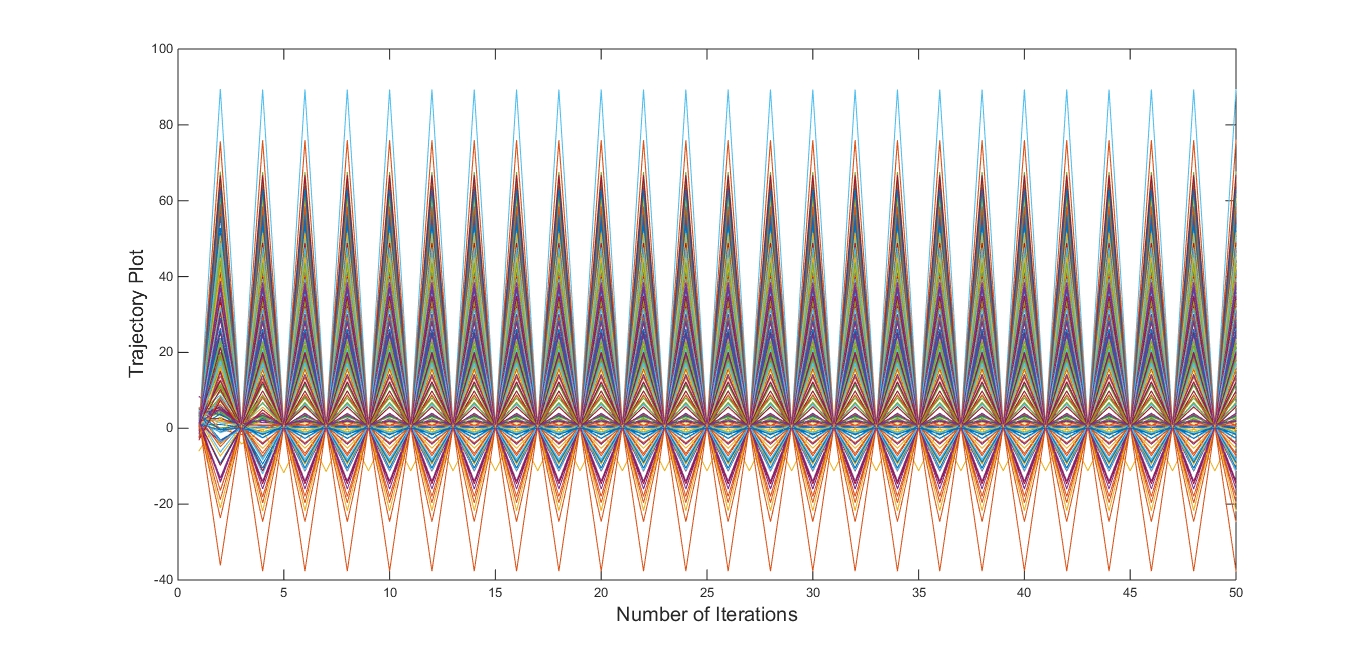}
      \end{tabular}
      }
\caption{Prime Period 2 Solutions for difference parameters and initial values}
      \begin{center}

      \end{center}
      \end{figure}

\begin{conjecture}
All solutions are periodic of period $2$ of the difference equation Eq.(\ref{equation:total-equationA}) for any initial values and for any values $\alpha$ and $\beta$ satisfying $\beta=\alpha+1$.
\end{conjecture}

\noindent
It is also confirmed from computational evidence that there are many higher order periodic solutions exist of the difference equation Eq.(\ref{equation:total-equationA}). Some of the computational observations are stated here.

\begin{table}[H]

\begin{tabular}{| m{4cm} || m{4.5cm} || m{1.5cm} || m{4cm}|}
\hline
\centering   \textbf{Parameters}: $\alpha$, $\beta$ &
\begin{center}
\textbf{Initial Values}
\end{center}
 &
\begin{center}
\textbf{Period}
\end{center} & \textbf{Condition:} $\abs{\beta} < \abs{\alpha+1}$\\
\hline
\centering $\alpha=0.0098 + 0.5323i$, $\beta=0.2794+ 0.9462i$ &
\begin{center}
$z_0=-0.5938 - 0.3212i$ $z_1=-1.2230 + 1.7184i$
\end{center} &
\begin{center}
7
\end{center} & $\abs{\alpha+1}=1.1415$ \& $\abs{\beta}=0.9866$
\\
\hline
\centering $\alpha=0.2021 + 0.4539i $, $\beta=0.4280 + 0.9660i$ &
\begin{center}
$z_0=-0.6653+ 2.800i$ $z_1=0.2991 + 0.6178i $
\end{center} &
\begin{center}
9
\end{center} & $\abs{\alpha+1}= 1.2849$ \& $\abs{\beta}=1.0566$
 \\
\hline
\centering $\alpha=89+55i$, $\beta=32+90i$ &
\begin{center}
$z_0=0.01272 + 0.6399i$ $z_1=-0.06293 + 0.02114i$
\end{center} &
\begin{center}
13
\end{center} & $\abs{\alpha+1}=105.4751$ \& $\abs{\beta}=95.5196$
  \\
\hline
\centering $\alpha=0.2776 + 0.65251i$, $\beta=\alpha-1$ &
\begin{center}
$z_0=-0.5909 + 3.2147i$ $z_1=0.02654 + 0.4106i$
\end{center} &
\begin{center}
36
\end{center} & $\abs{\alpha+1}=1.4346$ \& $\abs{\beta}=0.9735$
\\
\hline
\centering  $\alpha=0.0524+0.3234i$, $\beta=0.7996+ 0.6302i$ &
\begin{center}
$z_0=0.4480 + 0.2593i;$ $z_1=-1.03264 + 0.7212i;
$
\end{center} &
\begin{center}
40
\end{center} & $\abs{\alpha+1}=1.1010$ \& $\abs{\beta}= 1.0181$
\\
\hline

\end{tabular}
\caption{Higher order periodicities of the equation Eq.(\ref{equation:total-equationA}) for different choice of parameters and initial values.}
\label{Table:}
\end{table}
\noindent
It is observed that for any initial values and for the parameters $\alpha=(89, 55)$ and $\beta=(32, 90)$ all solutions are periodic and of period 13. The corresponding sequence trajectory and complex phase space plots are given in Fig. 3. The condition $\abs{\beta}< \abs{\alpha+1}$ is satisfied for all the cases in the Table 3. This confirms that the similar result in the context of real parameters and initial values is holding well in the complex set-up too.

\begin{figure}[H]
      \centering

      \resizebox{12cm}{!}
      {
      \begin{tabular}{c}
      \includegraphics [scale=5]{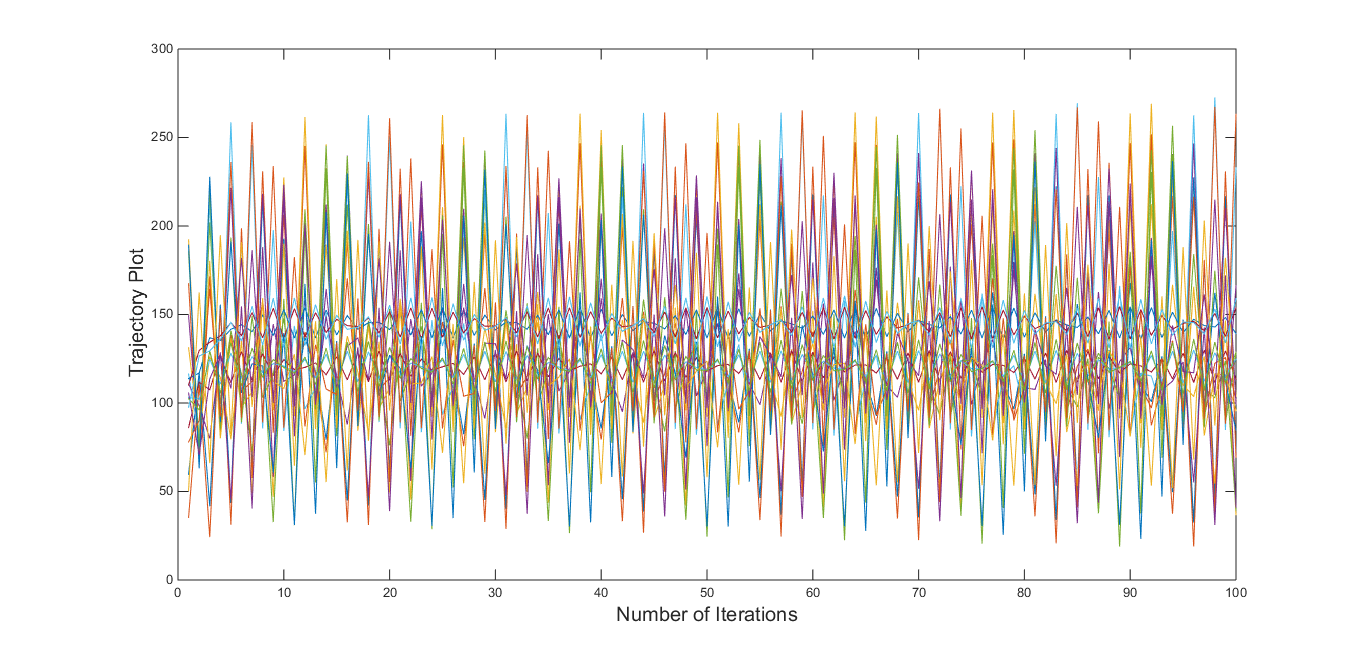}\\
      \includegraphics [scale=4]{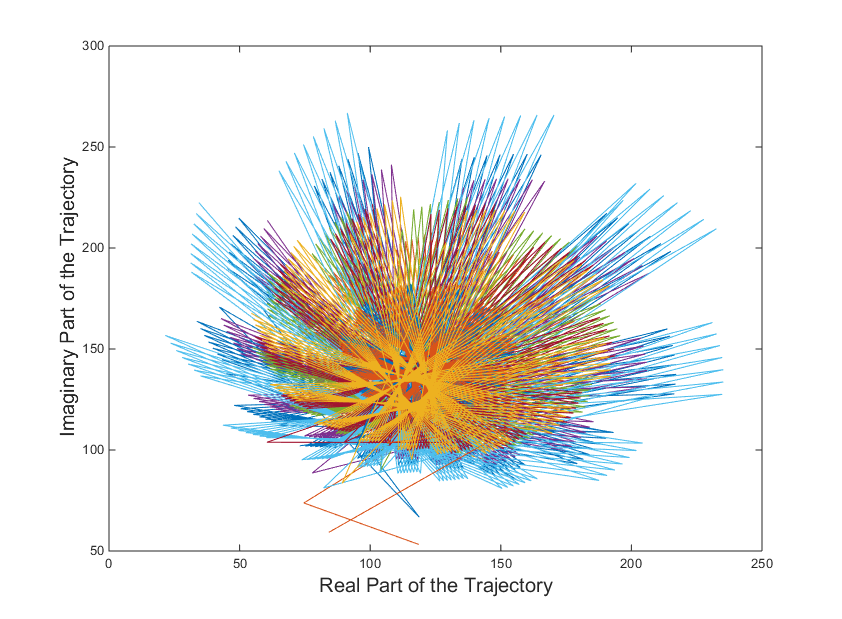}
\end{tabular}
      }
\caption{Periodic Trajectory and Orbit of Period 13.}
      \begin{center}

      \end{center}
      \end{figure}
\noindent
Here we consider 100 different initial values with the specified parameters $\alpha=(89, 55)$ and $\beta=(32, 90)$ and it is seen that all the solutions are periodic and of period 13. Different colors represent different trajectories(orbit) for different initial values in Fig.3.\\ \\

\noindent
The periodic trajectory and its corresponding orbit of period $36$ are given in the Fig.4 for the parameters $\alpha=0.2776 + 0.65251i$, $\beta=\alpha-1$.

\begin{figure}[H]
      \centering

      \resizebox{12cm}{!}
      {
      \begin{tabular}{c}
      \includegraphics [scale=5]{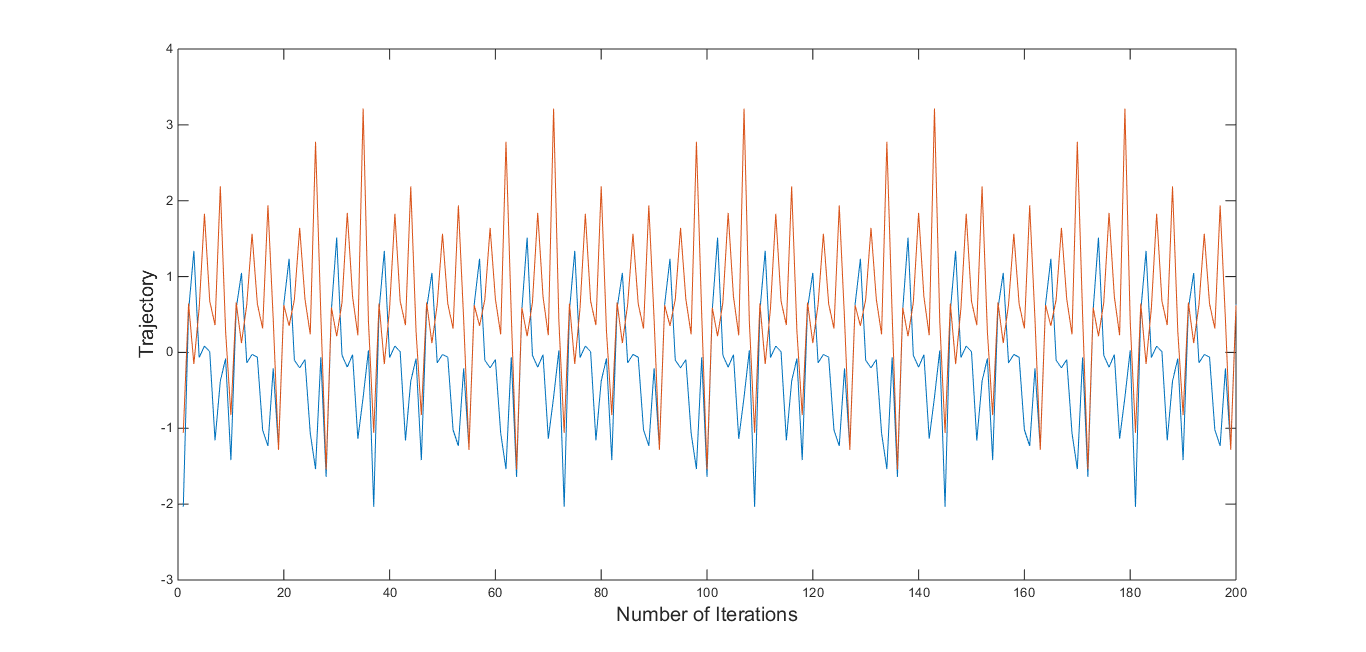}\\
      \includegraphics [scale=4]{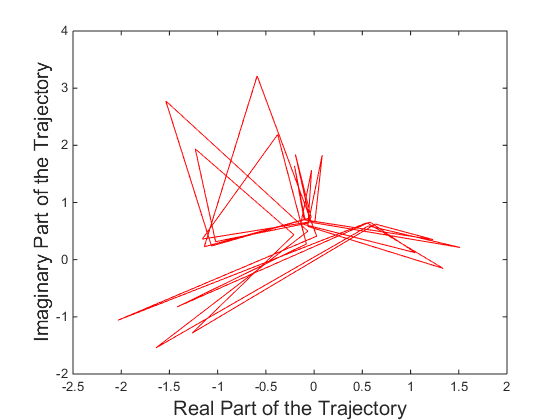}
\end{tabular}
      }
\caption{Periodic Trajectory and Orbit of Period 36.}
      \begin{center}

      \end{center}
      \end{figure}
\noindent
The periodic trajectory and its corresponding orbit  of period $40$ are given in the Fig.5 for the parameters $\alpha=0.05244 + 0.3234i$, $\beta=0.7996 + 0.6302i$.

\begin{figure}[H]
      \centering

      \resizebox{12cm}{!}
      {
      \begin{tabular}{c}
      \includegraphics [scale=5]{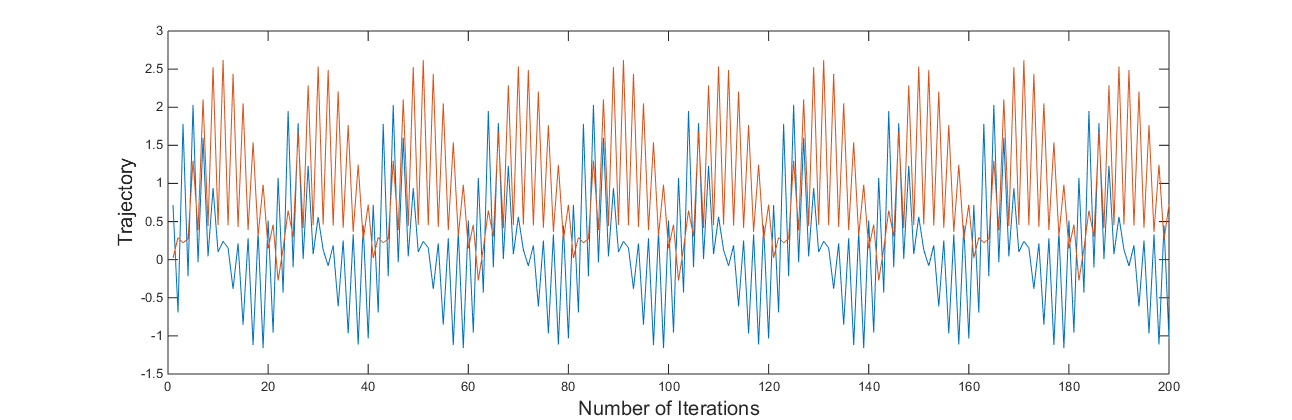}\\
      \includegraphics [scale=4]{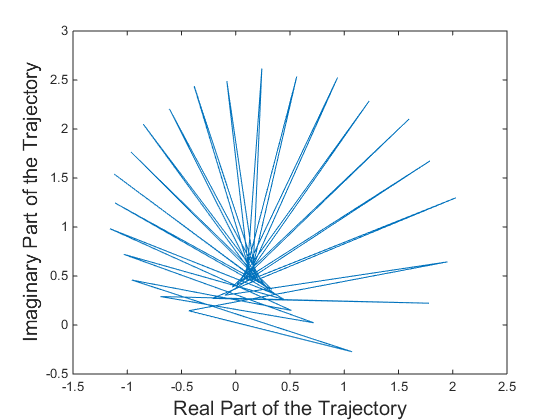}
\end{tabular}
      }
\caption{Periodic Trajectory and Orbit of Period 40.}
      \begin{center}

      \end{center}
      \end{figure}

\section{Chaotic Solutions}
This is something which is absolutely new feature of the dynamics of the difference equation Eq.(\ref{equation:total-equationA}) which did not arise in the real set up of the same difference equation. Computationally we have encountered some chaotic solution of the difference equation Eq.(\ref{equation:total-equationA}) for different parameters which are given in the following Table 4. \\
\noindent
The method of Lyapunov characteristic exponents serves as a useful tool to quantify chaos. Specifically Lyapunav exponents measure the rates of convergence or divergence of nearby trajectories. Negative Lyapunov exponents indicate convergence, while positive Lyapunov exponents demonstrate divergence and chaos. The magnitude of the Lyapunov exponent is an indicator of the time scale on which chaotic behavior can be predicted or transients decay for the positive and negative exponent cases respectively. In this present study, the largest Lyapunov exponent is calculated for a given solution of finite length numerically \cite{Wolf}.\\
\noindent
From computational evidence, it is arguable that for complex parameters $\alpha$ and $\beta$ which are stated in the following table the solutions are chaotic for every initial values.

\begin{table}[H]

\begin{tabular}{| m{5cm} || m{3cm} || m{6.5cm} |}
\hline
\centering   \textbf{Parameters} $\alpha$, $\beta$ &
\begin{center}
\textbf{Lyapunav exponent}
\end{center} &
\begin{center}
\textbf{$\abs{\alpha+1}$ $>$ $\abs{\beta}$}
\end{center}
\\
\hline
\centering $\alpha=0.2278 + 0.3210i$, $\beta=0.82956 + 0.8221i
$&
\begin{center}
$0.47153$
\end{center} & $\abs{\alpha+1}=1.2691$ \& $\abs{\beta}=1.1680$\\
\hline
\centering $\alpha=0.01365 + 0.37406i$, $\beta=0.92268 + 0.5464i$ &
\begin{center}
$1.062$
\end{center} & $\abs{\alpha+1}=1.0805$ \& $\abs{\beta}=1.0724$\\
\hline
\centering $\alpha=0.3377 + 0.2361i$, $\beta=0.3178 + 0.9844i$ &
\begin{center}
$0.6256$
\end{center} & $\abs{\alpha+1}=1.3584$ \& $\abs{\beta}=1.0345$\\
\hline
\centering $\alpha=0.1261 + 0.3001i$, $\beta=0.0021 + 0.9511i$ &
\begin{center}
$0.325$
\end{center} & $\abs{\alpha+1}=1.1654$ \& $\abs{\beta}=0.9511$\\
\hline
\centering $\alpha=0.1741 + 0.2446i$, $\beta=0.6409 + 0.8086i$ &
\begin{center}
$1.225$
\end{center} & $\abs{\alpha+1}=1.1993$ \& $\abs{\beta}=1.0318$\\
\hline
\centering $\alpha=0.1053 + 0.2682i$, $\beta=0.7638 + 0.8055i$ &
\begin{center}
$0.645$
\end{center} & $\abs{\alpha+1}=1.1374$ \& $\abs{\beta}=1.1101$\\
\hline
\centering $\alpha=0.0007 + 0.2836i$, $\beta=0.5508 + 0.8709i$ &
\begin{center}
$1.2678$
\end{center} & $\abs{\alpha+1}=1.0401$ \& $\abs{\beta}=1.0305$\\
\hline
\end{tabular}
\caption{Chaotic solutions of the equation Eq.(\ref{equation:total-equationA}) for different choice of parameters and initial values.}
\label{Table:}
\end{table}

\noindent
The chaotic trajectory plots including corresponding complex plots are given the following Fig. 6.

\begin{figure}[H]
      \centering

      \resizebox{14cm}{!}
      {
      \begin{tabular}{c}
      \includegraphics [scale=5]{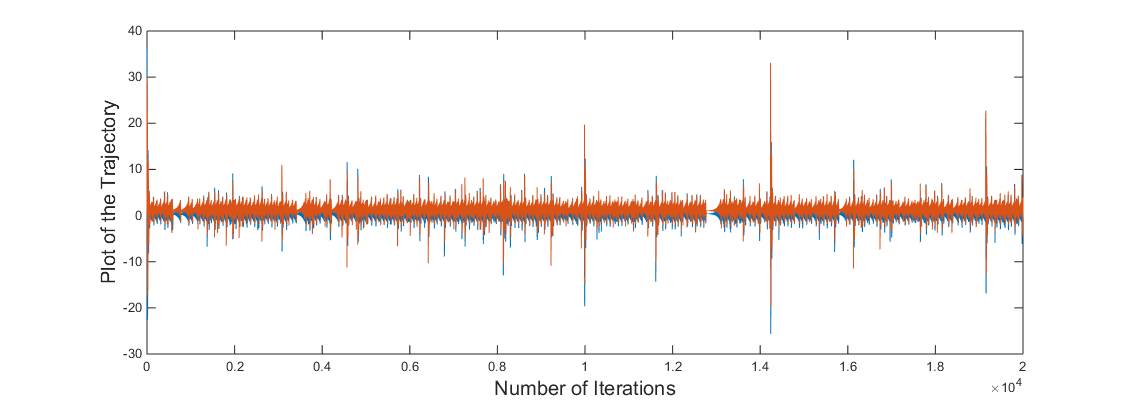}\\
      \includegraphics [scale=5]{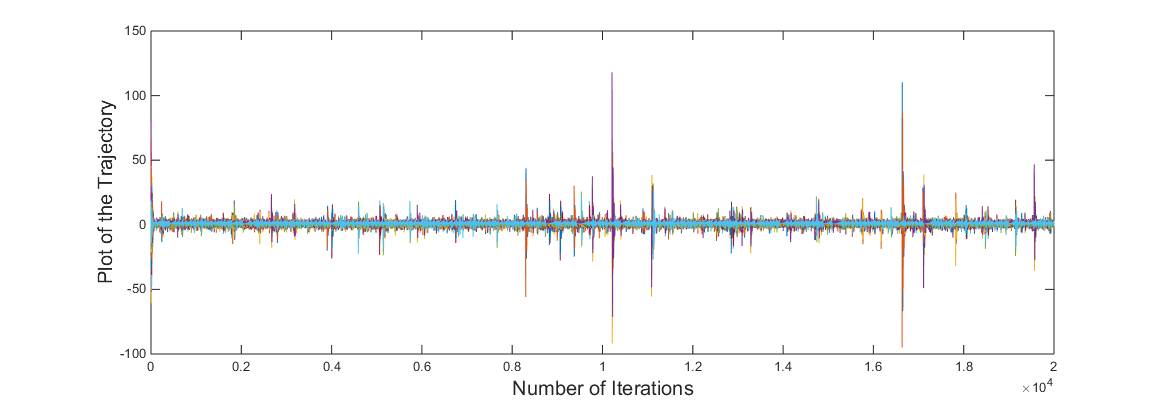}\\
      \includegraphics [scale=5]{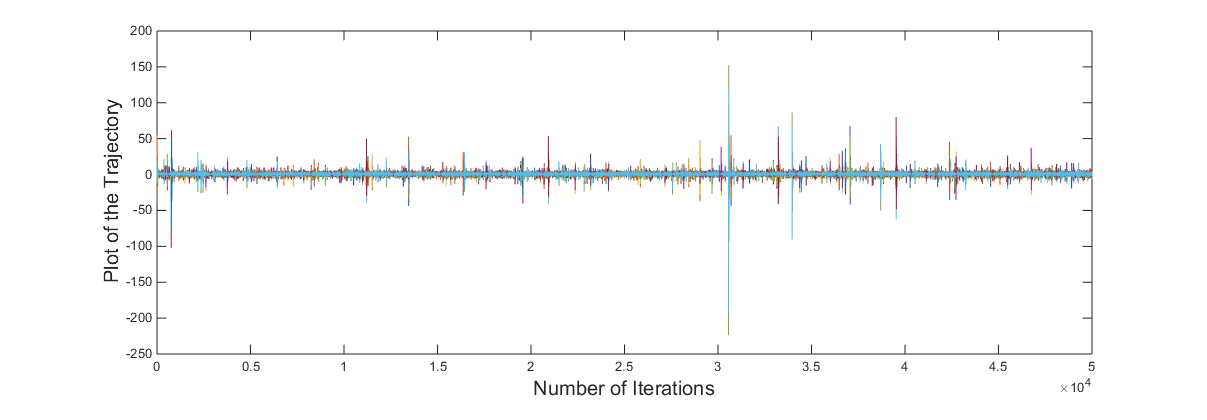}
            \end{tabular}
      }
\caption{Chaotic Trajectories of the equation Eq.(\ref{equation:total-equationA}) of the first row parameters as stated in Table 4.}
      \begin{center}

      \end{center}
      \end{figure}

\noindent
In the Fig. 6, for each of the four cases ten different initial values are taken and plotted in the left and in the right corresponding complex plots are given. From the Fig. 6, it is evident that for the four different cases the basin of the chaotic attractor is neighbourhood of the centre $(0, 0)$ of complex plane.\\

\noindent
In all the seven example cases in the Table 4, it is seen that the condition $\abs{\alpha+1}>\abs{\beta}$ is satisfied. based on this observation, the following conjectures has been made.\\

\begin{conjecture}
The chaotic solutions exist for the Eq.(\ref{equation:total-equationA}) if only if the condition $\abs{\alpha+1}>\abs{\beta}$ is satisfied.
\end{conjecture}

\begin{figure}[H]
      \centering

      \resizebox{17cm}{!}
      {
      \begin{tabular}{c}
      \includegraphics [scale=4]{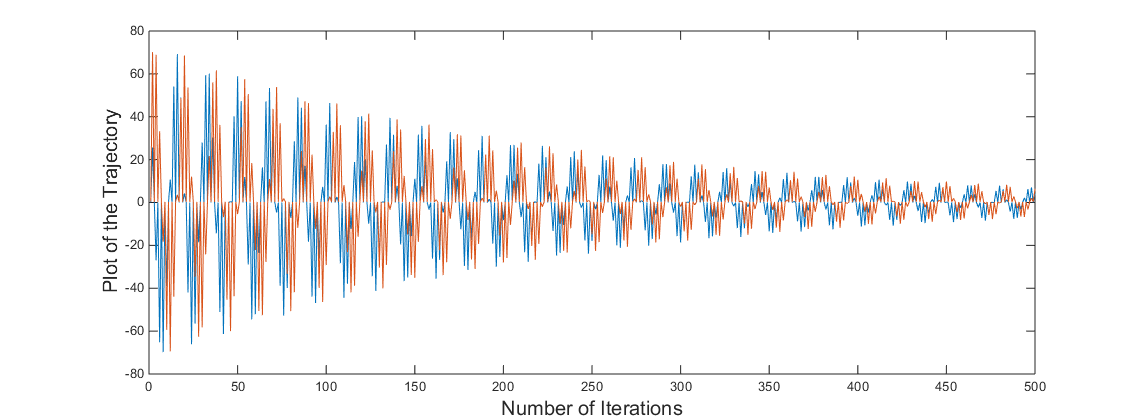}
      \includegraphics [scale=4]{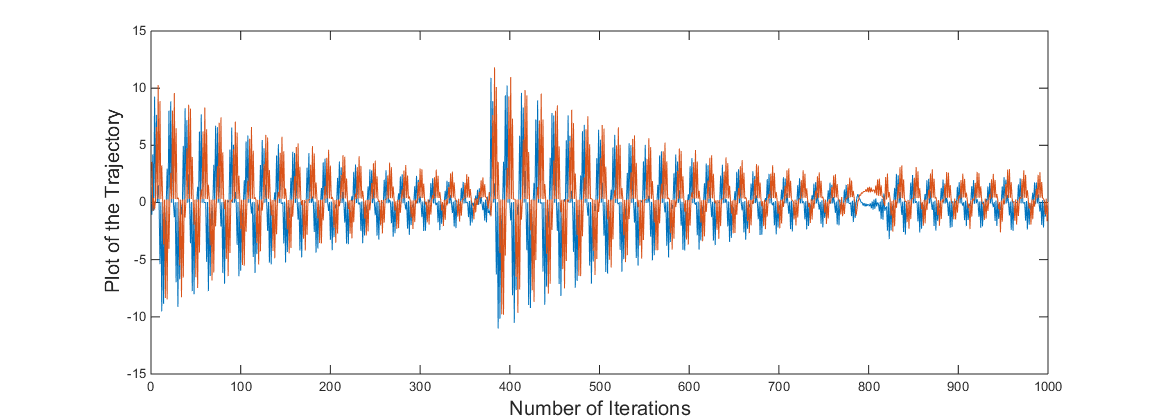}\\
      \includegraphics [scale=4]{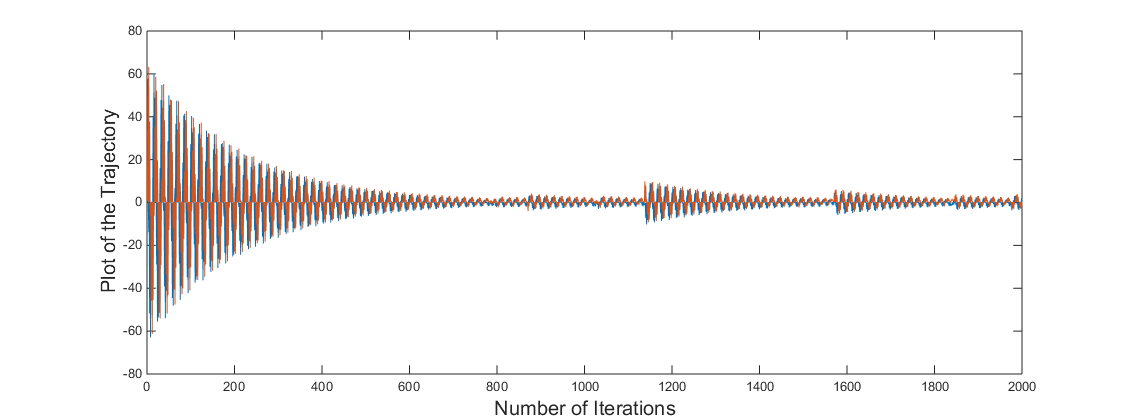}
      \includegraphics [scale=4]{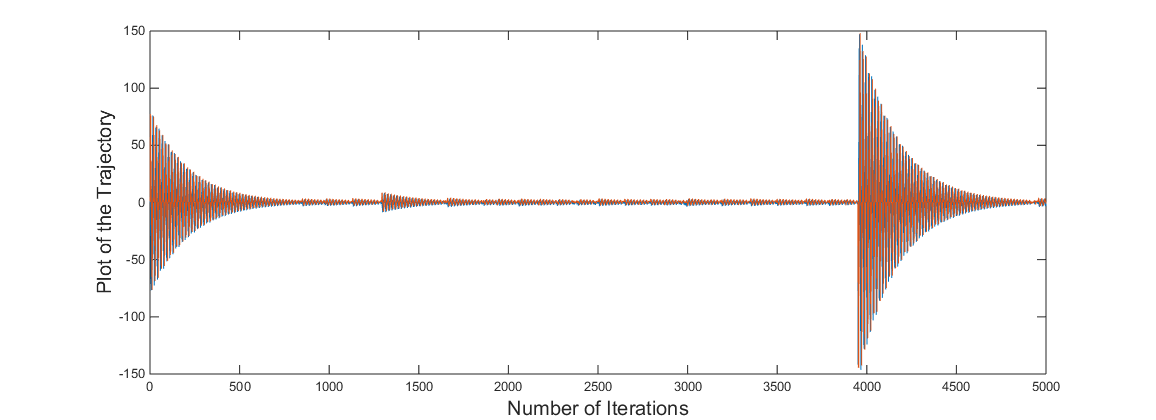}\\
      \includegraphics [scale=4]{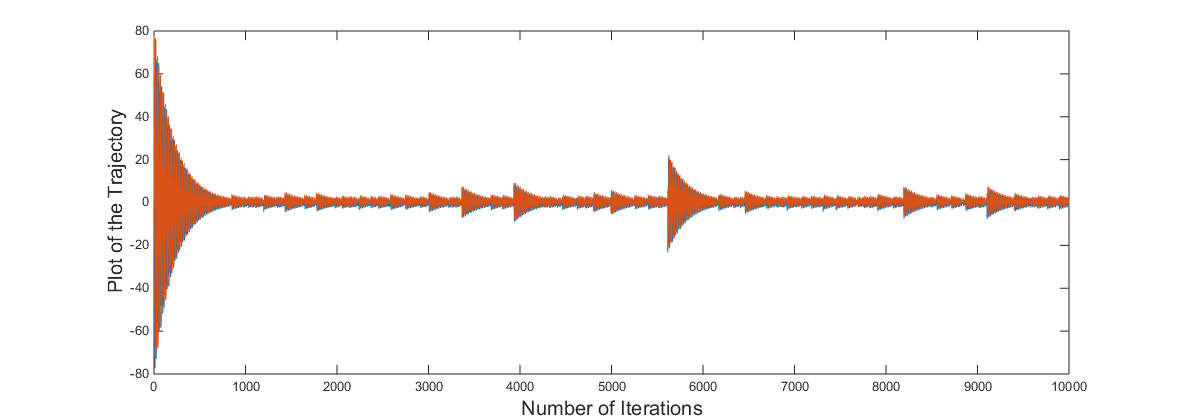}
      \includegraphics [scale=4]{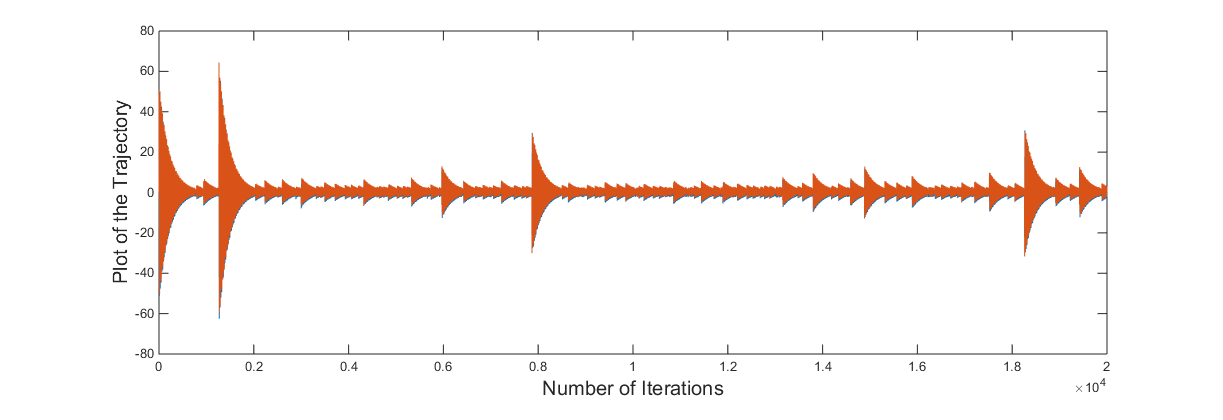}\\
      \includegraphics [scale=4]{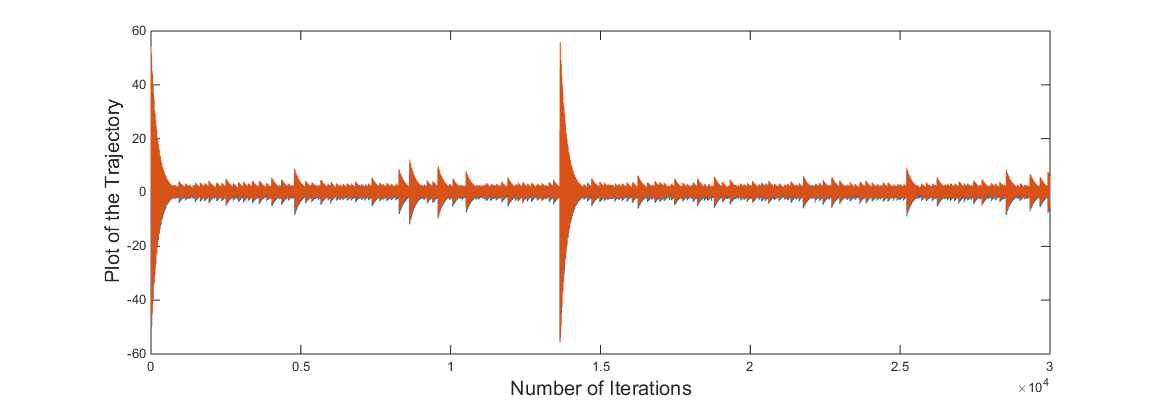}
      \includegraphics [scale=4]{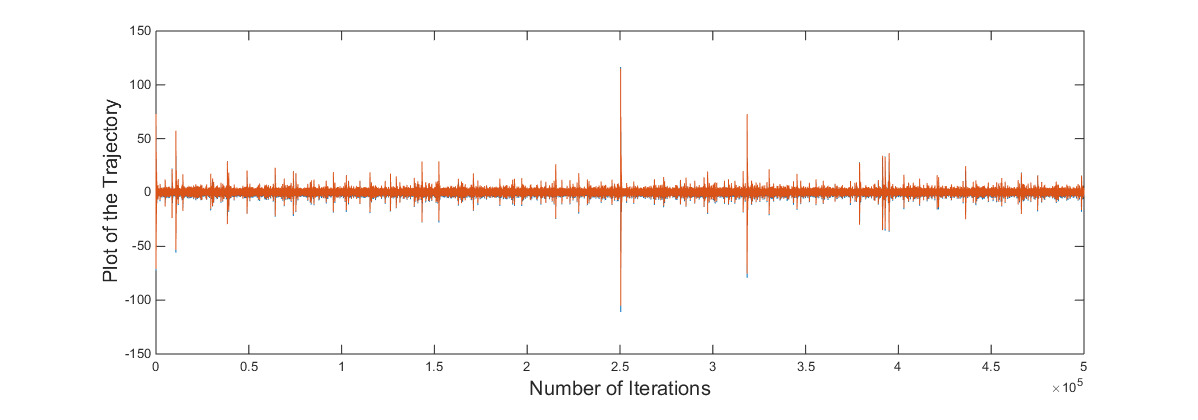}\\

            \end{tabular}
      }
\caption{Fractal-like Chaotic Trajectories of the equation Eq.(\ref{equation:total-equationA}) of the last row parameters from top of Table 4.}
      \begin{center}

      \end{center}
      \end{figure}

\noindent
The chaotic trajectory of the parameters of the last row from top of the Table 4 is given in the Fig. 7. It is to be noted that the for the nine different number of iterations, trajectories are plotted. Each trajectory looks like a self similar fractal. But eventually over iterations the trajectory becomes chaotic.


\section*{Acknowledgement}
The author thanks \emph{\textcolor[rgb]{0.00,0.00,1.00}{Dr. Pallab Basu}} of ICTS, TIFR for discussions and suggestions.



\begin{thebibliography}{9}

\bibitem{S-H} Saber N Elaydi, Henrique Oliveira, J Manuel Ferreira and J F Alves, (2007) \emph{Discrete Dyanmics and Difference Equations,
Proceedings of the Twelfth International Conference on Difference Equations and Applications}, World Scientific Press.

\bibitem{C-L} Camouzis, E. and Ladas, G., (2002), Three trichotomy conjectures, \emph{Journal of Difference Equations and Applications} 8(5), 495–-500.

\bibitem{Wolf} A. Wolf, J. B. Swift, H. L. Swinney and J. A. Vastano,  Determining Lyapunov exponents from a time series {Physica D}, 126(1985), 285-317.

\bibitem{S-E} Sk. S. Hassan, E. Chatterjee, {Dynamics of the equation $\displaystyle{z_{n+1}=\frac{\alpha + \beta z_{n}}{A+z_{n-1}}}$ in the Complex Plane}, \emph{Accepted for publication in Cogent Mathematics, Taylor and Francis}, 2015.

\bibitem{K-L} M.R.S. Kulenovi$\acute{c}$ and G. Ladas, \emph{Dynamics of Second Order Rational Difference Equations; With Open Problems and Conjectures}, Chapman \& Hall/CRC Press, 2001.

\bibitem{Gibb} C. H. Gibbons, M. R. S. Kulenovic, G. Ladas \& H. D. Voulov (2002) On the Trichotomy Character of $x_{n+1}=\frac{\alpha + \beta x_{n}+\gamma x_{n-1}}{A+x_{n}}$, Journal of Difference Equations and Applications, 8(1), 75--92.


\end{thebibliography}
\end{document}